\documentclass[final,3p]{elsarticle}

\usepackage{amsmath,amsthm,amssymb,mathrsfs}
\usepackage{enumerate}
\usepackage{slashed}
\usepackage{txfonts,dsfont}
\usepackage{aliascnt}
\usepackage[breaklinks,colorlinks]{hyperref}

\makeatletter
\newcommand{\autorefcheckize}[1]{%
  \expandafter\let\csname @@\string#1\endcsname#1%
  \expandafter\DeclareRobustCommand\csname relax\string#1\endcsname[1]{%
    \csname @@\string#1\endcsname{##1}\wrtusdrf{##1}}%
  \expandafter\let\expandafter#1\csname relax\string#1\endcsname
}
\makeatother

\theoremstyle{plain}

\newtheorem{theorem}{Theorem}[section]

\newaliascnt{lem}{theorem}
\newtheorem{lem}[lem]{Lemma}
 \aliascntresetthe{lem}

\newaliascnt{cor}{theorem}
\newtheorem{cor}[cor]{Corollary}
 \aliascntresetthe{cor}

\newaliascnt{prop}{theorem}
\newtheorem{prop}[prop]{Proposition}
 \aliascntresetthe{prop}

\newtheorem{known}{Theorem}

\newtheorem{conj}{Conjecture}

\theoremstyle{remark}

\newtheorem{rem}{Remark}[section]

\theoremstyle{definition}
\newtheorem{defn}{Definition}[section]

\numberwithin{equation}{section}

\newcommand{\abs}[1]{\left\lvert#1\right\rvert}
\newcommand{\set}[1]{\left\{#1\right\}}
\newcommand{\hin}[2]{\left\langle#1,#2\right\rangle}

\newcommand*{\To}{\longrightarrow}
\newcommand*{\Rmn}[1]{\uppercase\expandafter{\romannueral#1}}
\newcommand*{\dif}{\mathop{}\!\mathrm{d}}

\DeclareMathOperator{\Div}{div}

\journal{arXiv}

\allowdisplaybreaks

\begin{document}

\begin{frontmatter}

\title{Rigidity of closed CSL submanifolds in the unit sphere \tnoteref{LS}}

\author[whu]{Yong Luo}
\ead{yongluo@whu.edu.cn}

\author[whu]{Linlin Sun \corref{sll1}}
\ead{sunll@whu.edu.cn}

\tnotetext[LS]{Yong Luo was supported by the National Natural Science Foundation of China (Grant Nos. 11501421, 11771339). Linlin Sun was supported by the National Natural Science Foundation of China (Grant No. 11801420)  and Fundamental Research Funds for the Central Universities (Grant No. 2042018kf0044).}
\address[whu]{School of Mathematics and Statistics \& Computational Science Hubei Key Laboratory, Wuhan University, 430072 Wuhan, China}

\cortext[sll1]{Corresponding author.}

\begin{abstract}
A contact stationary Legendrian submanifold (briefly, CSL submanifold) is a stationary point of the volume functional of Legendrian submanifolds in a Sasakian manifold. Much effort has been paid in the last two decades to construct examples of such manifolds, mainly by geometers using various geometric methods. But we have rare knowledge about their geometric properties till now. Recently, Y. Luo (\cite{ Luo2, Luo1}) proved that a closed CSL surface in $\mathbb{S}^5$ with the square length of its second fundamental form belonging to $[0,2]$ must be totally geodesic or be a flat minimal Legendrian torus, which generalizes a related gap theorem of minimal Legendrian surface due to Yamaguchi et al. (\cite{YKM}). In this paper, we will study the general dimensional case of this result.

\end{abstract}


\end{frontmatter}


\section{Intruduction}
\subsection{CSL submanifolds in a Sasakian manifold}
Let $(\bar M^{2n+1},\bar\alpha, \bar g_{\bar\alpha},\bar J)$ be a $(2n+1)$-dimensional contact metric manifold with contact structure $\bar\alpha$, associated metric $\bar g_{\bar\alpha}$ and almost complex structure $\bar J$. Assume that $(M,g)$ is an $n$-dimensional  compact Legendrian submanifold of $\bar M^{2n+1}$ with the metric $g$ induced from $\bar g_{\bar\alpha}$. The volume of $M$ is defined by $V(M)\coloneqq\int_{M}\dif\mu_{g},$ where $\dif\mu_{g}$ is the volume form of $g$. A \emph{contact stationary Legendrian submanifold} (briefly, CSL submanifold) of $\bar M^{2n+1}$ is a Legendrian submanifold of $\bar M^{2n+1}$ which is a stationary point of $V$ with respect to contact deformations. In other words, a CSL submanifold is a stationary point of variation of the volume functional among Legendrian submanifolds. The Euler-Lagrange equation for a CSL submanifold $M$ is (\cite{CLU, Ir})
\begin{align*}
\Div_g(\bar J\mathbf{H})=0,
\end{align*}
where $\Div_g$ is the divergence operator with respect to $g$ and $\mathbf{H}$ is the mean curvature vector of $M$ in $\bar M^{2n+1}.$

\begin{rem}
The notion of CSL submanifold was first defined by Iriyeh in \cite{Ir}  and Castro et al. in  \cite{CLU} independently, where they used the name of Legendrian minimal Legendrian submanifold and contact minimal Legendrian submanifold, respectively. In this paper we prefer to use the name of CSL submanifold.
\end{rem}

 The study of CSL submanifolds was motivated by the study of Hamiltonian minimal Lagrangian (briefly, HSL) submanifolds, which was first studied by Ou (\cite{Oh90, Oh93}). A HSL submanifold in a K\"ahler manifold is a Lagrangian submanifold which is a stationary point of the volume functional under Hamiltonian deformations. By \cite{Re}, Legendrian submanifolds in a Sasakian manifold $\bar M^{2n+1}$ can be seem as links of Lagrangian submanifolds in the cone $C\bar M^{2n+1}$, which is a K\"ahler manifold with proper metric and complex structure. In fact, a close relation between CSL submanifolds and HSL  submanifolds was found by Iriyeh \cite{Ir} and Castro et al. \cite{CLU}. Precisely, they independently proved that $C(M)$ is a HSL submanifold in $\mathbb{C}^n(n\geq2)$ if and only if $M$ is a CSL submanifold in $\mathbb{S}^{2n-1}$ and $M$ is a CSL submanifold in $\mathbb{S}^{2n+1}(n\geq1)$ if and only if $\Pi(M)$ is a HSL submanifold in $\mathbb{CP}^n$, where $\Pi:\mathbb{S}^{2n+1}\to \mathbb{CP}^n$ is the Hopf fibration.

From the definition we see that CSL submanifolds are natural generalization of minimal Legendrian submanifolds.  The study of (nonminimal) CSL submanifolds of $\mathbb{S}^{2n+1}$ was relatively recent endeavor. For $n=1$, by \cite{Ir}, CSL curves in $\mathbb{S}^3$ are the so called $(p,q)$ curves discovered by Schoen and Wolfson in \cite{SW}, where $p,q$ are relatively prime integers. For $n=2$, since harmonic 1-form on a 2-sphere  must be trivial, CSL 2-sphere in $\mathbb{S}^5$ must be minimal and so must be the equatorial 2-spheres by Yau's result (\cite{Yau}). There are a lot of contact stationary doubly periodic surfaces form $\mathbb{R}^2$ to $\mathbb{S}^5$ by lifting H\'elein and Romon's examples (\cite{HR02}) and more CSL surfaces (mainly tori) were constructed in \cite{BuC, Ha, HR05, Ir, Ma, MaS, Mi03, Mi08} etc.. And for general dimension,  examples were constructed in \cite{Bu, CHX, Do, DoH, JLS, Lee, Mi04, Oh93} etc..
\subsection{Gap phenomenon of closed minimal submanifolds in the unit sphere}
In the theory of minimal submanifolds, the following Simons' integral inequality and
pinching theorem due to Simons (\cite{Si}), Lawson (\cite{La}) and Chern et al. (\cite{CCK}) are well-known.
\begin{known}[Simons, Lawson, Chern-Do Carmo-Kobayashi]
Let $M^n$ be a compact minimal submanifold in a unit sphere $\mathbb{S}^{n+p}$ and $\mathbf{B}$ the second fundamental form of $M$ in $\mathbb{S}^{n+p}$. Then we have
\begin{align*}
\int_M\abs{\mathbf{B}}^2\left(\dfrac{n}{2-\frac{1}{p}}-\abs{\mathbf{B}}^2\right)d\mu\leq0.
\end{align*}
In particular, if $0\leq \abs{\mathbf{B}}^2\leq \frac{n}{2-\frac{1}{p}},$
then either $\abs{\mathbf{B}}^2=0$ or $\abs{\mathbf{B}}^2=\frac{n}{2-\frac{1}{p}}$ and $M$ is the Clifford hypersurface or the Veronese surface in $\mathbb{S}^4$.
\end{known}

\begin{rem}
Usually we call the number $\frac{n}{2-\frac{1}{p}}$ the first gap of minimal submanifolds in a sphere because Chern conjectured that the set of numbers which are the square length of second fundamental form of compact minimal submanifolds in a sphere is discrete.
\end{rem}

From the classification of compact minimal submanifolds with $\abs{\mathbf{B}}^2=\frac{n}{2-\frac{1}{p}}$ we see that the first gap is not optimal except when $p=1$ or $n=p=2$.
It is interesting to sharpen the first gap for other cases of higher codimension.
In this direction there are many studies by several authors (see \cite{BKSS, Ga, Sh}) and finally Li-Li (\cite{LiLi}) and Chen-Xu (\cite{CX})
independently proved the following theorem.
\begin{known}[Li-Li, Chen-Xu]\label{thm:Li-Li}
Let $M^n$ be a compact minimal submanifold in a unit sphere. Assume that $\abs{\mathbf{B}}^2\leq\frac{2n}{3}$, then either $\abs{\mathbf{B}}^2=0$ and $M$ is totally geodesic or $n=2, \abs{\mathbf{B}}^2=\frac{4}{3}$ and $M$ is the Veronese surface in $\mathbb{S}^4$.
\end{known}
\begin{rem}
\autoref{thm:Li-Li} was generalized a little bit by Lu in \cite{Lu}.
\end{rem}
Legendrian submanifold is a special class of submanifolds, whose tangent bundle is isometric with its normal bundle.  Hence one may hope to solve the first gap problem for such class of submanifolds. This was done when $n=2$ (see \cite{YKM}).
 \begin{known}[Yamaguchi-Kon-Miyahara]\label{legendrian mimimal}
 If $\Sigma$ is a closed minimal Legendrian surface of the unit sphere $\mathbb{S}^5$ and $0\leq \abs{\mathbf{B}}^2\leq 2$, then $\abs{\mathbf{B}}^2$ is identically 0 or 2.
 \end{known}
 \begin{rem}
\autoref{legendrian mimimal} is inspired by Yau's Lagrangian version of this result (see \cite[Theorem 7]{Yau}).
 \end{rem}
 The higher dimensional case of this problem is largely remained open and one may see \cite{DV, YMI} for related results.
\subsection{Main results}
Besides effort made to obtain the existence of CSL submanifolds, people are also  interested in  understanding the properties of these examples. See \cite{HM, Ka} and \cite{Ono} for progress in this direction. To understand the geometry of CSL submanifolds and inspired by the first gap problem of closed minimal submanifolds in the unit sphere, Luo (\cite{Luo2, Luo1}) studied the first gap problem of CSL surfaces and get the following result.
\begin{known}[Luo]\label{thm:Luo}
Let $\Sigma$ be a closed contact stationary Legendrian surface in $\mathbb{S}^5$. Assume that $0\leq \abs{\mathbf{B}}^2\leq 2,$
then $\Sigma$ is either totally geodesic or $\abs{\mathbf{B}}^2=2$ and $\Sigma$ is a flat minimal Legendrian torus.
\end{known}
\begin{rem}
A flat minimal Legendrian torus in $\mathbb{S}^5$ must be a generalized Clifford torus, which also is a minimal Calabi torus stated in the appendix. For  details we refer to \cite[page 853]{Ha}.
\end{rem}
 The study toward the first gap problem for submanifolds satisfying a fourth order quasi-elliptic nonlinear equation was first carried out by Li. In \cite{Li1, Li2} and \cite{Li02}, Li proved several gap theorems for Willmore submanifolds in a sphere.

In this paper we are aiming to further study this kind of problem. We will not only generalize \autoref{thm:Luo} in dimension 2, but also prove such result in higher dimensions. The main results of this manuscript are the following:
\begin{theorem}\label{thm:main0}
Suppose $M^n (n\geq2)$ is a closed contact stationary Legendrian  submanifold of $\mathbb{S}^{2n+1}$ and
\begin{align}\label{eq:basic}
\abs{\mathbf{B}}^2\leq\dfrac{(n-1)(n+2)}{n}+\dfrac{n^2+3n-2}{2n^2}\abs{\mathbf{H}}^2-\dfrac{(n-1)(n-2)\abs{\mathbf{H}}\sqrt{4n+\abs{\mathbf{H}}^2}}{2n^2}.
\end{align}
\begin{enumerate}
\item If $n=2$, then $M$ is either totally geodesic or a Calabi torus stated in the appendix.
\item If $n\geq3$, then $M$ is either minimal or a Calabi product Legendrian immersion of a totally geodesic Legendrian immersion and a point, stated in the appendix.
\end{enumerate}
\end{theorem}

\begin{theorem}\label{thm:main}
If $M^n (n\geq3)$ is a closed contact stationary Legendrian  submanifold of $\mathbb{S}^{2n+1}$ and
\begin{align}\label{eq:main}
\abs{\mathbf{B}}^2\leq\dfrac{4(n-1)}{n}+\dfrac{3n-2}{n^2}\abs{\mathbf{H}}^2,
\end{align}
then $M$ is totally geodesic.
\end{theorem}
\begin{rem} According to examples of Calabi product Legendrian immersion of a totally geodesic Legendrian immersion and a point, we see that both \autoref{thm:main0} and \autoref{thm:main} are optimal.
\end{rem}

Since Hamiltonian minimal submanifolds in $\mathbb{CP}^n$ could be seen as  CSL submanifolds in $\mathbb{S}^{2n+1}$ by the Hopf projection, which is a local isometric map, our results could be translated to related results for Hamiltonian minimal submanifolds in $\mathbb{CP}^n$, which we would like to omit here.

\textbf{Organization:} In section 2 we give some preliminaries on the Sasakian geometry, CSL submanifolds in a sphere and prove several important lemmas which will be useful in the remaining sections. In section 3 and 4 we give a complete proof of \autoref{thm:main0}. Actually in section 3 we get  stronger results in the surface case. \autoref{thm:main} is proved in section 5. In section 6 we prove more results and propose several conjectures. In the Appendix , we state the examples which are not only used in the statement of our theorems, but also illustrate that both \autoref{thm:main0} and \autoref{thm:main} are optimal.
\section{Preliminaries}
In this section we recall some basic material from contact geometry and submanifold geometry. For more information we refer to \cite{Bl, Xi}.
\subsection{Contact Manifolds}
\begin{defn}
A contact manifold $\bar M$ is an odd dimensional manifold with a one form $\bar\alpha$ such that $\bar\alpha\wedge(\dif\bar\alpha)^n\neq0$, where $\dim \bar M=2n+1$.
\end{defn}
Assume now that $(\bar M,\bar \alpha)$ is a given contact manifold of dimension $2n+1$. Then $\bar\alpha$ defines a $2n$-dimensional vector bundle over $\bar M$, where the fibre at each point $p\in\bar M$ is given by
$$\bar\xi_p=\ker\bar\alpha_p.$$
Sine $\bar\alpha\wedge (\dif\bar\alpha)^n$ defines a volume form on $\bar M$, we see that
$$\bar\omega\coloneqq \dif\bar\alpha$$
is a closed nondegenerate 2-form on $\bar\xi\oplus\bar\xi$ and hence it defines a symplectic product on $\bar\xi$ such that $(\bar\xi,\bar\omega|_{\bar\xi\oplus\bar\xi})$ becomes a symplectic vector bundle. A consequence of this fact is that there exists an almost complex bundle structure
$$\tilde{J}:\bar\xi\to\bar\xi$$
compatible with $\dif\bar\alpha$, i.e. a bundle endomorphism satisfying:
\\(1) $\tilde{J}^2=-\mathrm{id}_{\bar\xi}$,
\\(2) $\dif\bar\alpha(\tilde{J}X,\tilde{J}Y)=\dif\bar\alpha(X,Y)$ for all $X,Y\in\bar\xi$,
\\(3) $\dif\bar\alpha(X,\tilde{J}X)>0$ for $X\in\bar\xi\setminus {0}$.

Since $\bar M$ is an odd dimensional manifold, $\bar\omega$ must be degenerate on $T\bar M$, and so we obtains a line bundle $\bar\eta$ over $\bar M$ with fibres
$$\bar\eta_p\coloneqq \set{V\in T_p\bar M | \omega(V,W)=0, \forall  W\in\bar\xi_p}.$$
\begin{defn}
The Reeb vector field $\bar{\mathbf{R}}$ is the section of $\bar\eta$ such that $\bar\alpha(\mathbf{\bar R})=1$.
\end{defn}

Thus $\bar\alpha$ defines a splitting of $T\bar M$ into a line bundle $\bar\eta$ with the canonical section $\mathbf{\bar R}$ and a symplectic vector bundle $(\bar\xi,\bar\omega|\bar\xi\oplus\bar\xi)$. We denote the projection along $\bar\eta$ by $\bar\pi$, i.e.
\begin{align*}
\bar\pi:T\bar M\to\bar \xi,\quad V\mapsto\bar\pi(V)\coloneqq V-\bar\alpha(V)\mathbf{\bar R}.
\end{align*}
Using this projection we extend the almost complex structure $\tilde{J}$ to a section $\bar J\in\Gamma(T^*\bar M\otimes T\bar M)$ by setting
$$\bar J(V)\coloneqq\tilde{J}(\pi(V)),$$
for $V\in T\bar M$.

We call $\bar J$ an almost complex structure of the contact manifold $\bar M$.
\begin{defn}
Let $(\bar M,\bar \alpha)$ be a contact manifold, a submanifold $M$ of $(\bar M,\bar\alpha)$ is called an isotropic submanifold if $T_xM\subseteq\bar\xi_x$ for all $x\in M$.
\end{defn}
For algebraic reasons the dimension of an isotropic submanifold of a $(2n+1)$-dimensional contact manifold can not be bigger than $n$.
\begin{defn}
An isotropic submanifold $M\subseteq(\bar M,\bar\alpha)$ of maximal possible dimension $n$ is called a Legendrian submanifold.
\end{defn}
\subsection{Sasakian manifolds}
Let $(\bar M,\bar \alpha)$ be a contact manifold, with the almost complex structure $\bar J$ and Reeb field $\mathbf{\bar R}$. A Riemannian metric $\bar g_{\bar \alpha}$ defined on $\bar M$ is said to be associated, if it satisfies the following three conditions:
\\(1) $\bar g_{\bar\alpha}(\mathbf{\bar R},\mathbf{\bar R})=1$,
\\(2) $\bar g_{\bar \alpha}(V,\mathbf{\bar R})=0$, $\forall  V\in\bar\xi$,
\\(3) $\bar\omega(V,\bar JW)=\bar g_{\bar\alpha}(V,W)$, $\forall V,W\in\bar\xi$.\

We should mention here that on any contact manifold there exists an associated metric on it.

Sasakian manifolds are the odd dimensional analogue of K\"ahler manifolds. They are defined as follows.
\begin{defn}
A contact manifold $(\bar M,\bar\alpha)$ with an associated metric $\bar g_{\bar\alpha}$ is called Sasakian, if the cone $C\bar M$ equipped with the following extended metric $\bar{\bar{g}}$
\begin{align*}
(C\bar M,\bar{\bar{g}})=\left(\mathbb{R}_+\times \bar M,dr^2+r^2\bar g_{\bar\alpha}\right)
\end{align*}
is K\"ahler w.r.t the following canonical almost complex structure $\bar{\bar J}$ on $TC\bar M=\mathbb{R}\oplus\langle\mathbf{\bar R}\rangle\oplus\bar\xi:$
$$\bar{\bar J}(r\partial r)=\mathbf{\bar R},\quad \bar{\bar J}(\mathbf{\bar R})=-r\partial r.$$
Furthermore if $\bar g_{\bar\alpha}$ is Einstein, $\bar M$ is called a Sasakian  Einstein manifold.
\end{defn}
We recall several lemmas which are well known in Sasakian geometry. These lemmas will be used in the next section.
\begin{lem}\label{Reeb}
Let $(\bar M,\bar\alpha,\bar g_{\bar\alpha},\bar J)$ be a Sasakian manifold. Then
\begin{align*}
\bar{\nabla}_X\mathbf{\bar R}=\bar JX.
\end{align*}
for $X,Y\in TM$, where $\bar{\nabla}$ is the Levi-Civita connection on $(\bar M,\bar g_{\bar\alpha})$.
\end{lem}
\begin{lem}\label{mean curvature form}
Let $M$ be a Legendrian submanifold in a Sasakian Einstein manifold $(\bar M,\bar\alpha,\bar g_{\bar\alpha},\bar J)$, then the mean curvature form $\bar\omega(\mathbf{H},\cdot)|_M$ defines a closed one form on $M$.
\end{lem}
For a proof of this lemma we refer to \cite[Proposition A.2]{Le} or \cite[lemma 2.8]{Sm}. In fact they proved this result under a weaker assumption that
 $(\bar M,\bar\alpha,\bar g_{\bar\alpha},\bar J)$ is a weakly Sasakian Einstein manifold, where weakly Einstein means that $\bar g_{\bar\alpha}$ is Einstein
 only when restricted to the contact hyperplane $\ker\bar\alpha$.
\begin{lem}
Let $M$ be a Legendrian submanifold in a Sasakian manifold $(\bar M,\bar\alpha,\bar g_{\bar\alpha},\bar J)$ and $\mathbf{B}$ be the second fundamental form of $M$ in $\bar M$. Then we have
\begin{align*}
\bar g_{\bar\alpha}(\mathbf{B}(X,Y),\mathbf{\bar R})=0,
\end{align*}
for any $X,Y\in TM$.
\end{lem}

In particular this lemma implies that the mean curvature $\mathbf{H}$ of $M$ is orthogonal to the Reeb field $\mathbf{\bar R}$.
\begin{lem}\label{commute of J}
For any $Y,Z\in \ker\bar\alpha$, we have
\begin{align*}
\bar g_{\bar\alpha}(\bar{\nabla}_X(\bar JY),Z)=\bar g_{\bar\alpha}(\bar J\bar{\nabla}_XY,Z).
\end{align*}
\end{lem}

A most canonical example of Sasakian Einstein manifolds is the standard odd dimensional sphere $\mathbb{S}^{2n+1}$.

\emph{The standard sphere $\mathbb{S}^{2n+1}$.}
 Let $\mathbb{C}^{n+1}=\mathbb{R}^{2n+2}$ be the Euclidean space with coordinates $(x_1,...,x_{n+1},y_1,...,y_{n+1})$ and $\mathbb{S}^{2n+1}$ be the standard unit sphere in $\mathbb{R}^{2n+2}$. Define
$$\alpha_0=\dfrac{1}{2}\sum_{j+1}^{n+1}\left(x_j\dif y_j-y_j\dif x_j\right),$$
then
$$\bar\alpha\coloneqq\alpha_0|_{\mathbb{S}^{2n+1}}$$
defines a contact one form on $\mathbb{S}^{2n+1}$. Assume that $g_0$ is the standard metric on $\mathbb{R}^{2n+2}$ and $J$ is the standard complex structure of $\mathbb{C}^{n+1}$. We define
$$\bar g_{\bar\alpha}=g_0|_{\mathbb{S}^{2n+1}}, \bar J=J|_{\mathbb{S}^{2n+1}},$$
then $(\mathbb{S}^{2n+1},\bar\alpha,\bar g_{\bar\alpha},\bar J)$ is a Sasakian Einstein manifold with associated metric $\bar g_{\bar\alpha}$. Its contact hyperplane is characterized by
$$\ker\bar\alpha_x=\set{Y\in T_x\mathbb{S}^{2n+1}|\hin{Y}{J\mathbf{x}}=0}.$$

\subsection{CSL submanifolds in the unit sphere}
Assume $\phi:M^n\To\mathbb{S}^{2n+1}\subseteq \mathbb{C}^{n+1} $ is a Legendrian immersion. Let $\mathbf{\mathbf{B}}$ be the second fundamental form, $\mathbf{A}^{\nu}$ be the shape operator with respect to the norm vector $\nu\in T^{\bot}M$ and $\mathbf{H}$ be the mean curvature vector. The shape operator $\mathbf{A}^{\nu}$ is a symmetric operator on the tangent bundle and satisfies the following  Weingarten equations
\begin{align*}
\hin{\mathbf{B}(X,Y)}{\nu}=\hin{\mathbf{A}^{\nu}(X)}{Y},\quad\forall X, Y\in TM, \nu\in T^{\bot}M.
\end{align*}
The Gauss equations, Codazzi equations and Ricci equations are given by
\begin{align*}
R(X,Y,Z,W)=&\hin{X}{Z}\hin{Y}{W}-\hin{X}{W}\hin{Y}{Z}+\hin{\mathbf{B}(X,Z)}{\mathbf{B}(Y,W)}-\hin{\mathbf{B}(X,W)}{\mathbf{B}(Y,Z)},\\
\left(\nabla_{X}^{\bot}\mathbf{B}\right)(Y,Z)=&\left(\nabla_{Y}^{\bot}\mathbf{B}\right)(X,Z),\\
R^{\bot}(X,Y,\mu,\nu)=&\hin{\mathbf{A}^{\mu}(X)}{\mathbf{A}^{\nu}(Y)}-\hin{\mathbf{A}^{\mu}(Y)}{\mathbf{A}^{\nu}(X)},
\end{align*}
where $X, Y, Z, W\in TM, \mu, \nu\in T^{\bot}M$.

Let $\set{e_i}$ be a local orthonormal frame of $M$.  Then $\set{Je_i, J\phi}$ is a local orthonormal frame of the normal bundle $T^{\perp}M$, where $J$ is the complex structure of $\mathbb{C}^{n+1}$.
 Recall that $M$ is CSL if and only if
\begin{align*}
\Div(J\mathbf{H})=0.
\end{align*}
It is obvious that $M$ is CSL when $M$ is minimal.

Notice that for all $X, Y,Z\in\Gamma(TM)$, by \autoref{Reeb} we see
\begin{align*}
(\nabla_Z^{\perp}\mathbf{B})(X,Y)=&-J(\nabla_Z(J\mathbf{B}))(X,Y)+\hin{Z}{J\mathbf{B}(X,Y)}J\phi.
\end{align*}
Thus,
\begin{align*}
\nabla^{\perp}_X\mathbf{H}=&-J\nabla_X(J\mathbf{H})+\hin{X}{J\mathbf{H}}J\phi,\\
\Div(J\mathbf{H})=&-\sum_{i=1}^n\hin{\nabla^{\perp}_{e_i}\mathbf{H}}{Je_i}.
\end{align*}
As an immediate consequence, there is no  closed  non-minimal CMC Legendrian submanifold in $\mathbb{S}^{2n+1}$. Moreover $M$ is CSL iff
\begin{align*}
\sum_{i=1}^n\hin{\nabla^{\perp}_{e_i}\mathbf{H}}{Je_i}=0.
\end{align*}

Set $\sigma_{ijk}\coloneqq\hin{\mathbf{B}(e_i,e_j)}{Je_k}$ and $\mu_j\coloneqq\hin{\mathbf{H}}{Je_j}=\sum_{i=1}^n\sigma_{iij}(1\leq i,j,k\leq n)$, then
\begin{align*}
\abs{\mathbf{B}}^2=\abs{\sigma}^2,\quad\abs{\mathbf{H}}^2=\abs{\mu}^2,\\
\abs{\nabla^{\perp}\mathbf{B}}^2=\abs{\nabla\sigma}^2+\abs{\sigma}^2,\quad\abs{\nabla^{\perp}\mathbf{H}}^2=\abs{\nabla\mu}^2+\abs{\mu}^2.
\end{align*}
Moreover by \autoref{commute of J}, the Codazzi equation and \autoref{mean curvature form} we have
\begin{align*}
\sigma_{ijk}=\sigma_{jik}=\sigma_{ikj},\quad\sigma_{ijk,l}=\sigma_{ijl,k},\\
\dif\mu=0,\quad\delta\mu=\Div\left(J\mathbf{H}\right).
\end{align*}
Therefore we have
\begin{lem}
 $M$ is CSL iff $\mu$ is a harmonic $1$-form iff  $J\mathbf{H}$ is a harmonic vector field.
\end{lem}
By using the Bochner formula for harmonic vector fields  (cf. \cite{Jo}), we get
\begin{lem}\label{lem:bochner}
If $M$ is CSL, then
 \begin{align*}
\dfrac12\Delta\abs{\mathbf{H}}^2=&\abs{\nabla\left(J\mathbf{H}\right)}^2+Ric\left(J\mathbf{H},J\mathbf{H}\right).
\end{align*}
\end{lem}

From \autoref{lem:bochner} it is easy to see that we have
\begin{lem}\label{lem:Gauss}
If $\Sigma\subset\mathbb{S}^5$ is CSL and non-minimal, then the zero set of $\mathbf{H}$ is isolate and
\begin{align*}
\Delta\log\abs{\mathbf{H}}=\kappa
\end{align*}
provided $\mathbf{H}\neq0$, where $\kappa$ is the Gauss curvature of $\Sigma$.
\end{lem}

We will need the following Simons' identity (cf. \cite{Si}, see also \cite{CO, YMI}).
\begin{lem}[Simons' identity]
Assume that $M^n$ is a Legendrian submanifold in $\mathbb{S}^{2n+1}$. Then
\begin{equation}\label{eq:simons}
\begin{split}
\Delta\sigma_{ijk}\coloneqq&\sum_l\sigma_{ijk,ll}\\
=&\mu_{i,jk}-\mu_{i}\delta_{jk}-\mu_{j}\delta_{ik}+\sum_{s,t}\sigma_{ijt}\sigma_{tks}\mu_{s}\\
&+(n+1)\sigma_{ijk}+2\sum_{l,s,t}\sigma_{isl}\sigma_{jlt}\sigma_{kts}-\sum_{l,s,t}\sigma_{tli}\sigma_{tls}\sigma_{jks}
-\sum_{l,s,t}\sigma_{tlj}\sigma_{tls}\sigma_{iks}-\sum_{l,s,t}\sigma_{tlk}\sigma_{tls}\sigma_{ijs}.
\end{split}
\end{equation}
Consequently,
\begin{align*}
\Delta\mu_k\coloneqq\sum_{l}\mu_{k,ll}=&\sum_{i}\mu_{i,ik}+(n-1)\mu_{k}+\sum_{s,t}\sigma_{tsk}\mu_{t}\mu_{s}-\sum_{l,s,t}\sigma_{tlk}\sigma_{tls}\mu_{s}.
\end{align*}
\end{lem}
\begin{proof}
The Ricci identity yields
\begin{align*}
\sigma_{ijk,lm}=\sigma_{ijk,ml}+\sum_{t}\sigma_{tjk}R_{tilm}+\sum_{t}\sigma_{itk}R_{tjlm}+\sum_{t}\sigma_{ijt}R_{tklm}.
\end{align*}
Therefore,
\begin{align*}
\Delta\sigma_{ijk}=&\sum_{l}\sigma_{ijk,ll}\\
=&\sum_{l}\sigma_{ijl,kl}\\
=&\sum_{l}\sigma_{ijl,lk}+\sum_{l,t}\sigma_{tjl}R_{tikl}+\sum_{l,t}\sigma_{itl}R_{tjkl}+\sum_{l,t}\sigma_{ijt}R_{tlkl}\\
=&\mu_{i,jk}+\sum_{l,t}\sigma_{tjl}R_{tikl}+\sum_{l,t}\sigma_{itl}R_{tjkl}+\sum_{l,t}\sigma_{ijt}R_{tlkl}.
\end{align*}
Thus,
\begin{align*}
\Delta\sigma_{ijk}=&\mu_{i,jk}+\sum_{l,t}\sigma_{tjl}\left(\delta_{tk}\delta_{il}-\delta_{tl}\delta_{ik}+\sigma_{tks}\sigma_{ils}-\sigma_{tls}\sigma_{iks}\right)\\
&+\sum_{l,t}\sigma_{til}\left(\delta_{tk}\delta_{jl}-\delta_{tl}\delta_{jk}+\sigma_{tks}\sigma_{jls}-\sigma_{tls}\sigma_{jks}\right)\\
&+\sum_{l,t}\sigma_{ijt}\left((n-1)\delta_{tk}+\sigma_{tks}\sigma_{lls}-\sigma_{tls}\sigma_{lks}\right)\\
=&\mu_{i,jk}+\sigma_{ijk}-\mu_{j}\delta_{ik}+\sum_{l,s,t}\sigma_{tjl}\left(\sigma_{tks}\sigma_{ils}-\sigma_{tls}\sigma_{iks}\right)\\
&+\sigma_{ijk}-\mu_{i}\delta_{jk}+\sum_{l,s,t}\sigma_{til}\left(\sigma_{tks}\sigma_{jls}-\sigma_{tls}\sigma_{jks}\right)\\
&+(n-1)\sigma_{ijk}+\sum_{l,s,t}\sigma_{ijt}\left(\sigma_{tks}\mu_{s}-\sigma_{tls}\sigma_{lks}\right)\\
=&\mu_{i,jk}-\mu_{i}\delta_{jk}-\mu_{j}\delta_{ik}+\sum_{s,t}\sigma_{ijt}\sigma_{tks}\mu_{s}\\
&+(n+1)\sigma_{ijk}+2\sum_{l,s,t}\sigma_{tjl}\sigma_{tks}\sigma_{ils}-\sum_{l,s,t}\sigma_{tjl}\sigma_{tls}\sigma_{iks}
-\sum_{l,s,t}\sigma_{til}\sigma_{tls}\sigma_{jks}-\sum_{l,s,t}\sigma_{tls}\sigma_{lks}\sigma_{ijt}.
\end{align*}
\end{proof}

\section{Rigidity results for closed CSL surfaces in the unit sphere}
In this section, we assume $\Sigma\subset\mathbb{S}^5$ is a closed CSL surface.

\begin{lem}[cf. \cite{Yau}]\label{lem:bochner_B}If $\Sigma$ is minimal and non totally geodesic , then the zero set of $\mathbf{B}$ is isolate and
\begin{align*}
\Delta\log\abs{\mathbf{B}}=3\kappa
\end{align*}
provided $\mathbf{B}\neq0$.
\end{lem}

\begin{cor}If $\Sigma\subset \mathbb{S}^5$ is a closed minimal Legendre surface with constant Gauss curvature, then  $\Sigma$ is either totally geodesic or a Calabi torus stated in the appendix.
\end{cor}
\begin{proof}By Gauss equation, $2\kappa=2-\abs{\mathbf{B}}^2$, we know that $\abs{\mathbf{B}}^2$ is a constant. According to \autoref{lem:bochner_B}, we know that either $\mathbf{B}\equiv0$ or $\kappa\equiv0$. Thus $\Sigma$ is either the totally geodesic sphere or a flat minimal Legendrian torus.
\end{proof}

More generally, we have
\begin{prop}\label{prop1}
Assume that $\Sigma\subset\mathbb{S}^5$ is a closed nonminimal Legendre surface with $\nabla\left(J\mathbf{H}\right)=0$. Then $\Sigma$ is a Calabi torus stated in the appendix.
\end{prop}

\begin{proof}
Denote $e_1=\frac{J\mathbf{H}}{\abs{\mathbf{H}}}$ and $\set{e_1,e_2}$ be the global orthonormal frame of $T\Sigma$. We consider the function $f\coloneqq 3\sigma_{111}-2\mu_1=\sigma_{111}-2\sigma_{122}$ where $\mu_1=\abs{\mathbf{H}}$ is a positive constant. The Simons' identity \eqref{eq:simons} gives
\begin{align*}
\dfrac13\Delta f=&-2\mu_{1}+\sum_{t}\sigma_{11t}^2\mu_{1}+3\sigma_{111}+2\sum_{l,s,t}\sigma_{1sl}\sigma_{1lt}\sigma_{1ts}-3\sum_{l,s,t}\sigma_{tl1}\sigma_{tls}\sigma_{11s}\\
=&f+\sum_{t}\sigma_{11t}^2\mu_{1}+2\sum_{l,s,t}\sigma_{1sl}\sigma_{1lt}\sigma_{1ts}-3\sum_{l,t}\sigma_{1tl}^2\sigma_{111}
-3\sum_{l,t}\sigma_{1tl}\sigma_{2tl}\sigma_{112}.
\end{align*}
On one hand, notice that $0=\mu_2=\sigma_{112}+\sigma_{222}$, we have
\begin{align*}
\sum_{l,t}\sigma_{1tl}\sigma_{2tl}=&\sigma_{111}\sigma_{211}+2\sigma_{112}\sigma_{212}+\sigma_{122}\sigma_{222}\\
=&\left(\sigma_{111}+\sigma_{122}\right)\sigma_{112}.
\end{align*}
On the other hand, assume the eigenvalues of the symmetric matrix $\left(\sigma_{1tl}\right)_{1\leq t,l\leq2}$ are $c_1, c_2$, then
\begin{align*}
\sum_{s,l,t}\sigma_{1sl}\sigma_{1lt}\sigma_{1ts}=&c_1^3+c_2^3\\
=&\left(c_1+c_2\right)\left(c_1^2+c_2^2-c_1c_2\right)\\
=&\left(\sigma_{111}+\sigma_{122}\right)\left(\sum_{l,t=1}^2\sigma_{1tl}^2-\left(\sigma_{111}\sigma_{122}-\sigma_{112}^2\right)\right).
\end{align*}
Thus,
\begin{align*}
&2\sum_{l,s,t}\sigma_{1sl}\sigma_{1lt}\sigma_{1ts}-3\sum_{l,t}\sigma_{1tl}^2\sigma_{111}-3\sum_{l,t}\sigma_{1tl}\sigma_{2tl}\sigma_{112}\\
=&\left(2\sigma_{122}
-\sigma_{111}\right)\sum_{l,t=1}^2\sigma_{1tl}^2-2\mu_1\sigma_{111}\sigma_{122}-\mu_1\sigma_{112}^2.
\end{align*}
Therefore,
\begin{align*}
\dfrac13\Delta f=&f+\sigma_{111}^2\mu_{1}-f\sum_{t,l=1}^2\sigma_{1tl}^2-2\mu_1\sigma_{111}\sigma_{122}\\
=&f\left[1+\sigma_{111}\mu_1-\sum_{t,l}\sigma_{1tl}^2\right]\\
=&f\Delta\mu_1\\
=&0.
\end{align*}
Consequently, $f$ is a constant. We conclude that $\sigma_{111}, \sigma_{122}$ both are constants. From \autoref{lem:Gauss} we see that $\kappa=0$, which implies from the Gauss equation $2\kappa=2+\abs{\mathbf{H}}^2-\abs{\mathbf{B}}^2$ that $\abs{\mathbf{B}}^2$ is a constant, we get that both $\sigma_{112}=-\sigma_{222}$ are constants. Up to now, we show that $\sigma$ is covariant constant (see also \cite{Yau}).

 We want to show that $\Sigma$ is a Calabi torus defined in the appendix.

At a point $p\in\Sigma$, we choose an orthonormal frame $\{e_1, e_2\}$ on $T_p\Sigma$ such that
$$\sigma(e_1,e_1,e_1)=\max_{\abs{X}=1, X\in T_pM}\sigma(X,X,X).$$ Then since $f(t)\coloneqq\sigma\left(\cos(t)e_1+\sin(t)e_2,\cos(t)e_1+\sin(t)e_2,\cos(t)e_1+\sin(t)e_2\right)$ achieves its maximum value at $t=0$, we see that $f'(0)=0$, which implies that $\sigma_{112}(p)=0$.  Since $\nabla\left(J\mathbf{H}\right)=0$, we see that the unit smooth orthogonal vector field of $J\mathbf{H}$, say $\upsilon$ is also parallel. Remember that we have proved $\sigma$ is a parallel 3-symmetric tensor. Assume that $(e_1,e_2)(p)=D(J\mathbf{H}, \upsilon)(p)$, where $D$ is a constant matrix. Then we extend $\{e_1,e_2\}$ to get a global orthonormal tangent vector frame on $\Sigma$ by $(E_1,E_2)\coloneqq D(J\mathbf{H}, \upsilon)$. Moreover, $E_1$ and $E_2$ are two unit parallel vector fileds on $\Sigma$.

We claim that
\begin{align}\label{eq:B2}
1+\sigma_{111}\sigma_{122}-\sigma_{122}^2=0.
\end{align}
 Assume that $\{\omega_1, w_2\}$ is the dual of $\{E_1,E_2\}$. Then the connection coefficient $\omega_{12}$ of $\Sigma$ equals zero since  $E_1$ is parallel. Then
\begin{align*}
0=&d\omega_{12}
\\=&-\omega_{13}\wedge\omega_{32}-\omega_{14}\wedge\omega_{42}+\omega_1\wedge\omega_2
\\=&\sum_{j,k}\sigma_{11j}\sigma_{12k}\omega_j\wedge\omega_k+\sum_{j,k}\sigma_{21j}\sigma_{22k}\omega_j\wedge\omega_k+\omega_1\wedge\omega_2
\\=&(1+\sigma_{111}\sigma_{122}-\sigma_{122}^2)\omega_1\wedge\omega_2.
\end{align*}

Due to \eqref{eq:B2}, we choose four nonzero constants $r_1, r_2, r_3, r_4$ such that $r_1^2+r_2^2=r_3^2+r_4^2=1$ and
\begin{align*}
\sigma_{111}=\dfrac{r_2}{r_1}-\dfrac{r_1}{r_2}, \quad \sigma_{122}=\dfrac{r_2}{r_1},\quad \sigma_{222}=\dfrac{1}{r_1}\left(\dfrac{r_4}{r_3}-\dfrac{r_3}{r_4}\right).
\end{align*}
Comparing  with the Calabi torus  define by $r_1, r_2, r_3, r_4$ stated in the appendix, we know that $\Sigma$ is locally isometric to the Calabi torus (see also \cite[Theorem 1.5]{LW11}). Since $\Sigma$ is closed, $\Sigma$ coincides with the Calabi torus.
\end{proof}

\begin{proof}[Proof the 2-dimensional case of \autoref{thm:main0}] By Gauss equation and assumption, we have $2\kappa=2+\abs{\mathbf{H}}^2-\abs{\mathbf{B}}^2\geq0.$ According to \autoref{lem:bochner}, since $Ric\left(J\mathbf{H},J\mathbf{H}\right)=\kappa \abs{\mathbf{H}}^2\geq0$, we know that
\begin{align*}
\nabla\left(J\mathbf{H}\right)\equiv0,\quad\kappa\abs{\mathbf{H}}^2\equiv0.
\end{align*}
If $M$ is minimal, then $M$ is either the totally geodesic sphere $\mathbb{S}^2$ or a flat minimal Legendrian torus by \autoref{legendrian mimimal}. If $M$ is not minimal, then the conclusion follows from \autoref{prop1}.
\end{proof}

\section{Rigidity results for closed CSL subamnifolds in the unit sphere}
In this section, we assume $M^n (n\geq3)$ is a closed CSL submanifolds in $\mathbb{S}^{2n+1}$.

Put
\begin{align*}
\sigma_{ijk}\coloneqq\mathring{\sigma}_{ijk}+\dfrac{1}{n+2}\left(\mu_i\delta_{jk}+\mu_j\delta_{ki}+\mu_k\delta_{ij}\right).
\end{align*}
Notice that $(\mathring{\sigma}_{ijk})$ is 3-symmetric and is trace free with any 2 symbols.

\begin{lem}\label{lem:main0}
Assume at some $p\in M$
\begin{align}\label{eq:basic1}
\abs{\mathbf{B}}^2\leq\dfrac{(n-1)(n+2)}{n}+\dfrac{n^2+3n-2}{2n^2}\abs{\mathbf{H}}^2-\dfrac{(n-1)(n-2)\abs{\mathbf{H}}\sqrt{4n+\abs{\mathbf{H}}^2}}{2n^2},
\end{align}
then at $p$ we have $Ric\left(J\mathbf{H},J\mathbf{H}\right)\geq0$. Moreover, if $\mathbf{H}\neq0$ then $Ric\left(J\mathbf{H},J\mathbf{H}\right)=0$ if and only if
\begin{align}\label{eq:basic_limit}
\mathbf{B}\left(J\mathbf{H},J\mathbf{H}\right)=\lambda_1 \abs{\mathbf{H}}\mathbf{H},\quad\mathbf{B}\left(J\mathbf{H},X\right)=\lambda_2\abs{\mathbf{H}} JX,\quad \mathbf{B}\left(X,Y\right)=\dfrac{\lambda_2}{\abs{\mathbf{H}}}\hin{X}{Y}\mathbf{H},\quad\forall X, Y\perp J\mathbf{H},
\end{align}
where $\lambda_1,\lambda_2$ satisfies
\begin{align*}
1+\lambda_1\lambda_2-\lambda_2^2=0.
\end{align*}
\end{lem}

\begin{proof}
Without loss of generality, assume $\mathbf{H}\neq0$ at $p$. Moreover, assume $\mu_1=\abs{\mathbf{H}}>0$ and $\mu_j=0$ for all $j>1$.
A direct calculation yields
\begin{align*}
Ric_{11}=&n-1+\sum_{j}\sigma_{11j}\mu_j-\sum_{j,k}\sigma_{1jk}^2\\
=&n-1+\dfrac{n-2}{n+2}\mathring{\sigma}_{111}\mu_1+\dfrac{2(n-1)}{(n+2)^2}\mu_1^2-\sum_{j,k}\mathring{\sigma}_{1jk}^2\\
\geq&n-1-\dfrac{n-2}{n+2}\abs{\mathring{\sigma}_{111}}\abs{\mu_1}+\dfrac{2(n-1)}{(n+2)^2}\mu_1^2-\sum_{j,k}\mathring{\sigma}_{1jk}^2.
\end{align*}
The equality holds if and only if
\begin{align}\label{eq:eq1}
\mathring{\sigma}_{111}\leq0.
\end{align}

Notice that
\begin{align*}
\sum_{i,j,k}\mathring{\sigma}_{ijk}^2=&\mathring{\sigma}_{111}^2+3\sum_{j=2}^n\mathring{\sigma}_{11j}^2+3\sum_{j=2}^n\mathring{\sigma}_{1jj}^2+6\sum_{2\leq j<k\leq n}\mathring{\sigma}_{1jk}^2+\sum_{i=2}^n\mathring{\sigma}_{iii}^2+3\sum_{2\leq i\neq j\leq n}\mathring{\sigma}_{ijj}^2+6\sum_{2\leq i<j<k\leq n}\mathring{\sigma}_{ijk}^2\\
\geq&\mathring{\sigma}_{111}^2+3\sum_{j=2}^n\mathring{\sigma}_{11j}^2+3\sum_{j=2}^n\mathring{\sigma}_{1jj}^2+6\sum_{2\leq j<k\leq n}\mathring{\sigma}_{1jk}^2+\dfrac{3}{n+1}\sum_{i=2}^n\left(\sum_{j=2}^n\mathring{\sigma}_{ijj}\right)^2+6\sum_{2\leq i<j<k\leq n}\mathring{\sigma}_{ijk}^2\\
=&\mathring{\sigma}_{111}^2+3\sum_{j=2}^n\mathring{\sigma}_{1jj}^2+\dfrac{3(n+2)}{n+1}\sum_{j=2}^n\mathring{\sigma}_{11j}^2+6\sum_{2\leq j<k\leq n}\mathring{\sigma}_{1jk}^2+6\sum_{2\leq i<j<k\leq n}\mathring{\sigma}_{ijk}^2\\
\geq&\dfrac{n+2}{n}\left(\mathring{\sigma}_{111}^2+\sum_{j=2}^n\mathring{\sigma}_{1jj}^2\right)+\dfrac{3(n+2)}{n+1}\sum_{j=2}^n\mathring{\sigma}_{11j}^2+6\sum_{2\leq j<k\leq n}\mathring{\sigma}_{1jk}^2+6\sum_{2\leq i<j<k\leq n}\mathring{\sigma}_{ijk}^2\\
\geq&\dfrac{n+2}{n}\sum_{j,k}\mathring{\sigma}_{1jk}^2\\
\geq&\dfrac{n+2}{n-1}\abs{\mathring{\sigma}_{111}}^2.
\end{align*}
And the equality holds if and only if (the assumption $n\geq3$ here is essential)
\begin{equation}\label{eq:eq2}
\begin{split}
\mathring{\sigma}_{11j}=0,\quad\mathring{\sigma}_{1jk}=0,\quad\mathring{\sigma}_{1jj}=\mathring{\sigma}_{1kk},\quad\forall 2\leq j<k\leq n,\\
\mathring{\sigma}_{ijk}=0,\quad\forall 2\leq i, j, k\leq n.
\end{split}
\end{equation}

Therefore, we obtain
\begin{align*}
Ric_{11}\geq&n-1-\dfrac{n-2}{n+2}\sqrt{\dfrac{n-1}{n+2}\sum_{i,j,k}\mathring{\sigma}_{ijk}^2}\abs{\mu_1}+\dfrac{2(n-1)}{(n+2)^2}\mu_1^2
-\dfrac{n}{n+2}\sum_{i,j,k}\mathring{\sigma}_{ijk}^2.
\end{align*}
Since the assumption \eqref{eq:basic1} is equivalent to
\begin{align*}
n-1-\dfrac{n-2}{n+2}\sqrt{\dfrac{n-1}{n+2}\sum_{i,j,k}\mathring{\sigma}_{ijk}^2}\abs{\mu_1}+\dfrac{2(n-1)}{(n+2)^2}\mu_1^2
-\dfrac{n}{n+2}\sum_{i,j,k}\mathring{\sigma}_{ijk}^2\geq0,
\end{align*}
we complete the proof of the first part.

 If $Ric\left(J\mathbf{H},J\mathbf{H}\right)=0$, then
 \eqref{eq:eq2} is equivalent to \eqref{eq:basic_limit} while \eqref{eq:eq1} is equivalent to
 \begin{align*}
 \lambda_1\leq\dfrac{3}{n+2}\abs{\mathbf{H}}.
 \end{align*}
 We conclude that
 \begin{align*}
 \lambda_2=\dfrac{1}{n-1}\left(\abs{\mathbf{H}}-\lambda_1\right)\geq\dfrac{1}{n+2}\abs{\mathbf{H}}\geq\dfrac{1}{3}\lambda_1.
 \end{align*}
According to \eqref{eq:basic_limit},
 \begin{align*}
 \abs{\mathbf{B}}^2=\lambda_1^2+3(n-1)\lambda_2^2,\quad\abs{\mathbf{H}}^2=\left(\lambda_1+(n-1)\lambda_2\right)^2.
 \end{align*}
 Thus,
 \begin{align*}
 \sum_{i,j,k}\mathring{\sigma}_{ijk}^2=\abs{\mathbf{B}}^2-\dfrac{3}{n+2}\abs{\mathbf{H}}^2=\dfrac{n-1}{n+2}\left(3\lambda_2-\lambda_1\right)^2,\quad\mu_1=\abs{\mathbf{H}}=\lambda_1+(n-1)\lambda_2.
 \end{align*}
 Now the equality in \eqref{eq:basic1} is equivalent to
\begin{align*}
&n-1-\dfrac{n-2}{n+2}\sqrt{\dfrac{n-1}{n+2}\sum_{i,j,k}\mathring{\sigma}_{ijk}^2}\abs{\mu_1}+\dfrac{2(n-1)}{(n+2)^2}\mu_1^2-\dfrac{n}{n+2}\sum_{i,j,k}\mathring{\sigma}_{ijk}^2=0,\\
\Longleftrightarrow\quad&\left(\sqrt{\dfrac{n+2}{n-1}\sum_{i,j,k}\mathring{\sigma}_{ijk}^2}+\dfrac{n-2}{2n}\abs{\mu_1}\right)^2=\dfrac{(n+2)^2}{4n^2}\left(\mu_1^2+4n\right),\\
\Longleftrightarrow\quad&\left(\left(3\lambda_2-\lambda_1\right)+\dfrac{n-2}{2n}\left(\lambda_1+(n-1)\lambda_2\right)\right)^2=\dfrac{(n+2)^2}{4n^2}\left(\left(\lambda_1+(n-1)\lambda_2\right)^2+4n\right),\\
\Longleftrightarrow\quad&\left((n+1)\lambda_2-\lambda_1\right)^2=\left(\lambda_1+(n-1)\lambda_2\right)^2+4n,
\end{align*}
which is equivalent to
\begin{align*}
1+\lambda_1\lambda_2-\lambda_2^2=0.
\end{align*}

\end{proof}

\begin{proof}[Proof of \autoref{thm:main0} when $n\geq3$]
From \autoref{lem:main0} we have $Ric\left(J\mathbf{H},J\mathbf{H}\right)\geq0$ and \autoref{lem:bochner} implies that $\nabla\left(J\mathbf{H}\right)=0$. Therefore we have either $\mathbf{H}=0$ or $\mathbf{H}\neq0$ and the second fundamental form of $M$ in $\mathbf{S}^{2n+1}$ is given by \eqref{eq:basic_limit} where $\lambda_1$ and $\lambda_2$ are two smooth functions satisfying
\begin{align*}
1+\lambda_1\lambda_2-\lambda_2^2=0.
\end{align*}

Now assume $\mathbf{H}\neq0$. Let $e_1=\frac{J\mathbf{H}}{\abs{\mathbf{H}}}$, then
\begin{align*}
\sigma_{111}=\lambda_1,\quad\sigma_{11j}=0,\quad\sigma_{1jk}=\lambda_2\delta_{jk},\quad\sigma_{ijk}=0,\quad\forall 2\leq i, j, k\leq n.
\end{align*}
According to Simons' identity \eqref{eq:simons}, we have
\begin{align*}
\Delta\lambda_1=&-2\mu_{1}+\sum_{t}\sigma_{11t}^2\mu_{1}+(n+1)\sigma_{111}+2\sum_{l,s,t}\sigma_{1sl}\sigma_{1lt}\sigma_{1ts}
-3\sum_{l,s,t}\sigma_{tl1}\sigma_{tls}\sigma_{11s}\\
=&-2\mu_{1}+\lambda_1^2\mu_{1}+(n+1)\lambda_1+2\left(\lambda_1^3+(n-1)\lambda_2^3\right)-3\left(\lambda_1^2+(n-1)\lambda_2^2\right)\lambda_1\\
=&-2(n-1)\lambda_2+(n-1)\lambda_1^2\lambda_2+(n-1)\lambda_1+2(n-1)\lambda_2^3-3(n-1)\lambda_2^2\lambda_1\\
=&(n-1)\left(\lambda_1-2\lambda_2\right)\left(1+\lambda_1\lambda_2-\lambda_2^2\right)\\
=&0.
\end{align*}
Hence $\lambda_1$ and $\lambda_2$ are constants. Therefore there exists two constants $r_1, r_2$ such that $r_1^2+r_2^2=1$ and
\begin{align*}
\lambda_1=\dfrac{r_2}{r_1}-\dfrac{r_1}{r_2},\quad\lambda_2=\dfrac{r_2}{r_1}.
\end{align*}

Comparing with the Calabi product Legendrian immersion of a totally geodesic Legendrian immersion and a point determined by $r_1, r_2$ stated in the appendix,
we may conclude that $M$ is locally isometric to the Calabi product Legendrian immersion of a totally geodesic Legendrian immersion and a point (see also \cite[Theorem 1.5]{LW11}).  Since $M$ is closed, we conclude that $M$ must be a Calabi product Legendrian immersion of a totally geodesic Legendrian immersion and a point. This completes the proof of \autoref{thm:main0}.
\end{proof}

As an application of \autoref{thm:main0}, we give a proof of \autoref{thm:main}. Firstly, we prove
\begin{theorem}\label{thm:epsilon}
Suppose $M^n(n\geq3)$ be a closed CSL submanifold of $\mathbb{S}^{2n+1}$ and for some $\varepsilon>0$ we have
\begin{align*}
\abs{\mathbf{B}}^2\leq&\dfrac{(n-1)(n+2)}{n}-\dfrac{(n-1)(n-2)\varepsilon}{n}+\left(\dfrac{n^2+3n-2}{2n^2}-\dfrac{(n-1)(n-2)}{4n^2}\left(\varepsilon+\dfrac{1}{\varepsilon}\right)\right)\abs{\mathbf{H}}^2.
\end{align*}

\begin{enumerate}
\item If $\varepsilon\geq1$, then $M$ is a minimal Legendrian immersion.
\item If $0<\varepsilon<1$, then $M$ is either a minimal Legendrian immersion or the Calabi product Legendrian immersion of the totally geodesic $\psi:\mathbb{S}^{n-1}\To\mathbb{S}^{2n-1}$ and a point.
\end{enumerate}
\end{theorem}
\begin{proof}By Young's  inequality, for every $\varepsilon>0$, we have
\begin{align*}
2\sqrt{\mathbf{H}}\sqrt{4n+\abs{\mathbf{H}}^2}\leq\dfrac{\abs{\mathbf{H}}^2}{\varepsilon}+\varepsilon\left(4n+\abs{\mathbf{H}}^2\right).
\end{align*}
The equality holds iff
\begin{align*}
\dfrac{\abs{\mathbf{H}}^2}{\varepsilon}=\varepsilon\left(4n+\abs{\mathbf{H}}^2\right).
\end{align*}
Therefore, under the assumption, we have \eqref{eq:basic}. Moreover, when $\varepsilon\geq1$, we have the strictly inequality.
\begin{enumerate}
\item When $\varepsilon\geq1$, according to \autoref{thm:main0}, we know that $M$ is minimal.
\item When $0<\varepsilon<1$, according to \autoref{thm:main0}, we know that $M$ is either a minimal Legendrian immersion or the Calabi product Legendrian immersion of the totally geodesic $\psi:\mathbb{S}^{n-1}\To\mathbb{S}^{2n-1}$ and a point.
\end{enumerate}
\end{proof}

\begin{theorem}\label{thm:main1}
If $M^n(n\geq3)$ be a closed CSL submanifold of $\mathbb{S}^{2n+1}$ and
\begin{align*}
\abs{\mathbf{B}}^2\leq&\dfrac{2(n+1)}{3}-\dfrac{n-17}{3(n+3)}\abs{\mathbf{H}}^2.
\end{align*}

\begin{enumerate}
\item If $n=3$, then $M$ is the totally geodesic Legendrian immersion.
\item If $n\geq4$, then $M$ is either the totally geodesic Legendrian immersion or is a nonminimal Calabi product Legendrian immersion of the totally geodesic $\psi:\mathbb{S}^{n-1}\To\mathbb{S}^{2n-1}$ and a point.
\end{enumerate}
\end{theorem}
\begin{proof}Choose $\varepsilon=\frac{n+3}{3(n-1)}(n\geq3)$ as in \autoref{thm:epsilon}. Since $\varepsilon\geq1$ iff $n=3$, when $n=3$, $M$ is minimal with $\abs{\mathbf{B}}^2\leq\frac{2(n+1)}{3}$.   It remains to consider the case that $n\geq3$,  $M$ is minimal and
\begin{align*}
\abs{\mathbf{B}}^2\leq\dfrac{2(n+1)}{3}.
\end{align*}

According to Simons' identity \eqref{eq:simons}, we have
\begin{align*}
\dfrac12\Delta\abs{\sigma}^2=&\abs{\nabla\sigma}^2+(n+1)\abs{\sigma}^2+2\sum_{i,j,k,l,s,t}\sigma_{isl}\sigma_{jlt}\sigma_{kts}\sigma_{ijk}-
3\sum_{i,j,k,l,s,t}\sigma_{tli}\sigma_{tls}\sigma_{jks}\sigma_{ijk}.
\end{align*}
Define
\begin{align*}
A_{i}=\left(\sigma_{ijk}\right)_{1\leq j, k\leq n},\quad 1\leq i\leq n,
\end{align*}
then (see \cite{YMI})
\begin{align*}
\dfrac12\Delta\abs{\sigma}^2=&\abs{\nabla\sigma}^2+(n+1)\abs{\sigma}^2-\sum_{i,j}\abs{[A_i,A_j]}^2-\sum_{i,j}\hin{A_i}{A_j}^2\\
\geq&\abs{\nabla\sigma}^2+(n+1)\abs{\sigma}^2-\dfrac{3}{2}\left(\sum_{i}\abs{A_i}^2\right)^2\\
=&\abs{\nabla\sigma}^2+(n+1)\abs{\sigma}^2-\dfrac{3}{2}\abs{\sigma}^4.
\end{align*}
By assumption, we have
\begin{align*}
\dfrac12\Delta\abs{\sigma}^2\geq\abs{\nabla\sigma}^2.
\end{align*}
Thus $\sigma\equiv0$ or $\abs{\sigma}\equiv\frac{2(n+1)}{3}$. Following the same argument of \cite{CX} or \cite{LiLi}, when $\abs{\sigma}\equiv\frac{2(n+1)}{3}$, we must have $n=2$, hence the last case can not happen. Therefore $M$ is totally geodesic.

\end{proof}

\begin{proof}[Proof of \autoref{thm:main}]
The condition \eqref{eq:main} in \autoref{thm:main} is just as the case $\varepsilon=1$ as in \autoref{thm:epsilon}. We conclude that $M$ is minimal and $\abs{\mathbf{B}}^2\leq\frac{4(n-1)}{n}$. Since $\frac{4(n-1)}{n}\leq\frac{2(n+1)}{3}$ always holds when $n\geq3$, by \autoref{thm:main1}, we finish the proof.

\end{proof}

\section{More results and discussions}
In this section we will get more results from \autoref{thm:epsilon}.

\begin{theorem}\label{thm:main3}
Suppose $M^n(n\geq3)$ be a closed CSL submanifold of $\mathbb{S}^{2n+1}$ and
\begin{align*}
\abs{\mathbf{B}}^2\leq
\begin{cases}
\frac{2(n+1)}{3},&3\leq n\leq 16,\\
2\left(\sqrt{3n-2}-1\right)& n\geq17.
\end{cases}
\end{align*}
\begin{enumerate}
\item If $3\leq n\leq 16$, then $M$ is the totally geodesic Legendrian immersion.
\item If $n\geq17$, then $M$ is either the totally geodesic Legendrian immersion or the Calabi product Legendrian immersion of the totally geodesic $\psi:\mathbb{S}^{n-1}\To\mathbb{S}^{2n-1}$ and a point.
\end{enumerate}
\end{theorem}
\begin{proof}Take $\varepsilon=\frac{\left(n-\sqrt{3n-2}\right)^2}{(n-1)(n-2)}$ as in \autoref{thm:epsilon}. Notice that $\varepsilon>0$ since $n\geq3$ and $\frac{2(n+1)}{3}<2\left(\sqrt{3n-2}-1\right)$ when $3\leq n\leq16$ we see that $\abs{\mathbf{B}}^2\leq \frac{2(n+1)}{3}<2\left(\sqrt{3n-2}-1\right)$, which implies that $M$ is minimal and hence totally geodesic by the same argument with \autoref{thm:main1}. When $n\geq17$, if $M$ is minimal, then $\abs{\mathbf{B}}^2\leq2\left(\sqrt{3n-2}-1\right)\leq\frac{2(n+1)}{3}$ we get $M$ is totally geodesic by the same argument with \autoref{thm:main1} again. Therefore we complete the proof.
\end{proof}

\begin{theorem}If $M^n(n\geq3)$ be a closed CSL submanifold of $\mathbb{S}^{2n+1}$ and
\begin{align*}
\abs{\mathbf{B}}^2\leq2+\dfrac{3}{n+1}\abs{\mathbf{H}}^2.
\end{align*}
Then $M$ is the totally geodesic Legendrian immersion.
\end{theorem}
\begin{proof}Take $\varepsilon=\frac{n+1}{n-1}$ as in \autoref{thm:epsilon}. Notice that $\varepsilon>1$ and hence $M$ is minimal. Therefore $\abs{\mathbf{B}}^2\leq 2\leq\frac{2(n+1)}{3}$. Then by a similar argument to \autoref{thm:main1} we complete the proof.
\end{proof}

\begin{rem}Under the assumption $n\geq4$ and
\begin{align*}
\abs{\mathbf{B}}^2<2+\dfrac{3}{n+3/2}\abs{\mathbf{H}}^2,
\end{align*}
Li-Wang (\cite{LW}) proved that the closed simply connected Legendrian submanifold $M^n$ in $\mathbb{S}^{2n+1}$ must be diffemorphic to $\mathbb{S}^n$. Under the assumption
\begin{align*}
\abs{\mathbf{B}}^2<\begin{cases}
6+\frac{3}{n+2/3}\abs{\mathbf{H}}^2,&n\geq5,\\
6+\frac{3}{4}\abs{\mathbf{H}}^2,&n=4,
\end{cases}
\end{align*}
Sun-Sun (\cite{SS}) proved that closed simply connected Legendrian submanifold $M^n$ in $\mathbb{S}^{2n+1}$ must be a topological sphere.
\end{rem}

At the end of this paper, we list some conjectures.
\begin{conj}Let $M^n (n\geq2)$ be a closed minimal Legendrian submanifold in $\mathbb{S}^{2n+1}$ and
\begin{align*}
\abs{\mathbf{B}}^2\leq\dfrac{(n-1)(n+2)}{n},
\end{align*}
then  $M$ is either totally geodesic or a minimal Calabi product Legendrian immersion of a totally geodesic Legendrian immersion and a point.
\end{conj}
This conjecture is equivalent to
\begin{conj}
If $M^n (n\geq2)$ is a closed contact stationary Legendrian  submanifold of $\mathbb{S}^{2n+1}$ and
\begin{align*}
\abs{\mathbf{B}}^2\leq\dfrac{(n-1)(n+2)}{n}+\dfrac{n^2+3n-2}{2n^2}\abs{\mathbf{H}}^2-\dfrac{(n-1)(n-2)\abs{\mathbf{H}}\sqrt{4n+\abs{\mathbf{H}}^2}}{2n^2}.
\end{align*}
Then $M$ is either totally geodesic or a Calabi product Legendrian immersion of a totally geodesic Legendrian immersion and a point.
\end{conj}
From \autoref{thm:main0}, we know that this conjecture is true for $n=2$.

For the first gap of the length of fundamental form of CSL submanifolds in the unit sphere, motived by \autoref{thm:main3}, we list a conjecture.
\begin{conj}If $M^n (n\geq2)$ be a closed  contact stationary Legendrian  submanifold of $\mathbb{S}^{2n+1}$ and
\begin{align*}
\abs{\mathbf{B}}^2\leq 2\left(\sqrt{3n-2}-1\right),
\end{align*}
then $M$ is either totally geodesic or a Calabi product Legendrian immersion of a totally geodesic Legendrian immersion and a point.
\end{conj}
\autoref{thm:Luo} and \autoref{thm:main3} claims that this conjecture is true for $n=2$ and $n\geq17$ respectively.

\appendix
\section{Examples}
\subsection{Calabi tori}
For every four nonzero real numbers $r_1, r_2, r_3, r_4$ with $r_1^2+r_2^2=r_3^2+r_4^2=1$, a Calabi torus is CSL surface in $\mathbf{S}^5$ defined as follows.
\begin{align*}
F:\Sigma\coloneqq&\mathbb{S}^1\times\mathbb{S}^{1}\To\mathbb{S}^5,\\
(t,s)\mapsto&\left(r_1r_3\exp\left(\sqrt{-1}\left(\dfrac{r_2}{r_1}t
+\dfrac{r_4}{r_3}s\right)\right),r_1r_4\exp\left(\sqrt{-1}\left(\dfrac{r_2}{r_1}t-\dfrac{r_3}{r_4}s\right)\right),
r_2\exp\left(-\sqrt{-1}\dfrac{r_1}{r_2}t\right)\right).
\end{align*}

Denote
\begin{align*}
\phi_1=\exp\left(\sqrt{-1}\left(\dfrac{r_2}{r_1}t+\dfrac{r_4}{r_3}s\right)\right),\quad\phi_2
=\exp\left(\sqrt{-1}\left(\dfrac{r_2}{r_1}t-\dfrac{r_3}{r_4}s\right)\right),\quad\phi_3=\exp\left(-\sqrt{-1}\dfrac{r_1}{r_2}t\right),
\end{align*}
then $F(t,s)=\left(r_1r_3\phi_1, r_1r_4\phi_2, r_2\phi_3\right)$. Since
\begin{align*}
\dfrac{\partial F}{\partial t}=\left(\sqrt{-1}r_2r_3\phi_1, \sqrt{-1}r_2r_4\phi_2, -\sqrt{-1}r_1\phi_3\right),\\
\dfrac{\partial F}{\partial s}=\left(\sqrt{-1}r_1r_4\phi_1, -\sqrt{-1}r_1r_3\phi_2, 0\right),
\end{align*}
the induced metric in $\Sigma$ is given by
\begin{align*}
g=\dif t^2+r_1^2\dif s^2.
\end{align*}
Let $E_1=\frac{\partial F}{\partial t}, E_2=\frac{1}{r_1}\frac{\partial F}{\partial s}$, then $\set{E_1, E_2, \nu_1=\sqrt{-1}E_1, \nu_2=\sqrt{-1}E_2, \mathbf{R}=\sqrt{-1}F}$ is a local orthonormal frames of $\mathbb{S}^5$ such that $\set{E_1,E_2}$ is a local orthonormal tangent frames and $\mathbf{R}$ is the Reeb field. A direct calculation yields
\begin{align*}
\dfrac{\partial\nu_1}{\partial t}=&\left(-\sqrt{-1}\dfrac{r_2^2r_3}{r_1}\phi_1, -\sqrt{-1}\dfrac{r_2^2r_4}{r_1}\phi_2, -\sqrt{-1}\dfrac{r_1^2}{r_2}\phi_3\right),\\
\dfrac{\partial\nu_1}{\partial s}=&\left(-\sqrt{-1}\dfrac{r_2r_3^2}{r_4}\phi_1, \sqrt{-1}\dfrac{r_2r_4^2}{r_3}\phi_2,0\right),\\
\dfrac{\partial\nu_2}{\partial t}=&\left(-\sqrt{-1}\dfrac{r_2r_4}{r_1}\phi_1, \sqrt{-1}\dfrac{r_2r_3}{r_1}\phi_2, 0\right),\\
\dfrac{\partial\nu_2}{\partial s}=&\left(-\sqrt{-1}\dfrac{r_4^2}{r_3}\phi_1,-\sqrt{-1}\dfrac{r_3^2}{r_4}\phi_2, 0\right),\\
\dfrac{\partial\mathbf{R}}{\partial t}=&\left(-r_2r_3\phi_1, -r_2r_4\phi_2, r_1\phi_3\right),\\
\dfrac{\partial\mathbf{R}}{\partial s}=&\left(-r_1r_4\phi_1, r_1r_3\phi_2,0\right).
\end{align*}
Hence,
\begin{align*}
\mathbf{A}^{\nu_1}=&-\Re\hin{\dif F}{\dif \nu_1}=\left(\dfrac{r_2}{r_1}-\dfrac{r_1}{r_2}\right)\dif t^2+r_1r_2\dif s^2,\\
\mathbf{A}^{\nu_2}=&-\Re\hin{\dif F}{\dif \nu_2}=2r_2\dif t\dif s+r_1\left(\dfrac{r_4}{r_3}-\dfrac{r_3}{r_4}\right)\dif s^2,\\
\mathbf{A}^{\mathbf{R}}=&0.
\end{align*}
Thus
\begin{align*}
\mathbf{H}=&\left(\dfrac{2r_2}{r_1}-\dfrac{r_1}{r_2}\right)\nu_1+\dfrac{1}{r_1}\left(\dfrac{r_4}{r_3}-\dfrac{r_3}{r_4}\right)\nu_2.
\end{align*}
Moreover $E_1$ and $E_2$ are two parallel tangent vector field. Under the orthonormal frames $\set{E_1,E_2}$, the second fundament form can be written as follows:
\begin{align*}
\mathbf{A}^{\nu_1}=\begin{pmatrix}\dfrac{r_2}{r_1}-\dfrac{r_1}{r_2}&0\\
0&\dfrac{r_2}{r_1}
\end{pmatrix},\quad \mathbf{A}^{\nu_2}=\begin{pmatrix}0&\dfrac{r_2}{r_1}\\
\dfrac{r_2}{r_1}&\dfrac{1}{r_1}\left(\dfrac{r_4}{r_3}-\dfrac{r_3}{r_4}\right)
\end{pmatrix},\quad\mathbf{A^{R}}=0.
\end{align*}

A direct calculation shows that
\begin{align*}
\kappa=2+\abs{\mathbf{H}}^2-\abs{\mathbf{B}}^2=0.
\end{align*}
It is obvious that $J\mathbf{H}$ is parallel. In particular, $\Sigma$ is CSL.  Moreover, $F$ is a minimal Legendrian surface iff $r_1=\pm\frac{\sqrt{6}}{3}, r_2=\pm\frac{\sqrt{3}}{3}, r_3=r_4=\pm\frac{\sqrt{2}}{2}$. In this case $\abs{\mathbf{B}}^2=2$ and the Gauss curvature of $F$ is 0, i.e. $F$ is a flat minimal Legendrian torus.
\subsection{Calabi product Legendrian immersions}

Let $F=\left(F^1,F^2,\dotsc, F^{n+1}\right):M^n\To\mathbb{S}^{2n+1}\subset\mathbb{C}^{n+1}$ be an isometric immersion. Then $F$ is a Legendrian immersion if and only if
\begin{align*}
\sum_{\alpha}F^{\alpha}_{i}\bar F^{\alpha}=0,\quad\sum_{\alpha}F^{\alpha}_{i}\bar F^{\alpha}_{j}=\sum_{\alpha}F^{\alpha}_{j}\bar F^{\alpha}_{i},\quad\forall i, j.
\end{align*}

Let $\gamma=(\gamma_1,\gamma_2):\mathbb{S}^1\To\mathbb{S}^{3}, t\mapsto\left(r_1\exp\left(\sqrt{-1}\frac{r_2}{r_1}t\right),r_2\exp\left(-\sqrt{-1}\frac{r_1}{r_2}t\right)\right)$ be a Legendre curve where $r_1,r_2$ are two nonzero constants satisfying $r_1^2+r_2^2=1$. Let $F=\left(F^1,F^2,\dotsc,F^n\right):M^{n-1}\To\mathbb{S}^{2n-1}$ be a Legendrian immersion. Then $\tilde F\coloneqq\left(\gamma^1F,\gamma^2\right):\tilde M\coloneqq\mathbb{S}^1\times M\To\mathbb{S}^{2n+1}$ is a Legendrian immersion. We call $\tilde F$ a \emph{Calabi product Legendrian immersion} of $F$ and a point.

The induced metric on $\tilde M$ is given by
\begin{align*}
\tilde g=\dif t^2+r_1^2g,
\end{align*}
where $g$ is the induced metric on $M$. Denote
\begin{align*}
E_1=&\left(\sqrt{-1}r_2\exp\left(\sqrt{-1}\dfrac{r_2}{r_1}t\right)F,-\sqrt{-1}r_1\exp\left(-\sqrt{-1}\dfrac{r_1}{r_2}t\right)\right)=\dif\tilde F\left(\dfrac{\partial}{\partial t}\right),\\
E_j=&\left(\exp\left(\sqrt{-1}\dfrac{r_2}{r_1}t\right)\dif F\left(e_j\right),0\right)=\dfrac{1}{r_1}\dif\tilde F\left(e_j\right),\quad j=2,\dotsc, n,
\end{align*}
where $\set{\dif F(e_j)}_{j=2}^n$ is a local orthonormal frames of  $TM$. We obtain a local orthonormal frames $\set{E_j}_{j=1}^n$ of $T\tilde M$. Then $\set{\nu_j\coloneqq\sqrt{-1}E_j, \sqrt{-1}\tilde F}$ is a local orthonormal frames of the normal bundle $T^{\bot}\tilde M$. A direct calculation yields
\begin{align*}
\mathbf{\tilde A}^{\nu_1}=&-\Re\set{\hin{\dif\tilde F}{\dif\nu_1}}=\left(\dfrac{r_2}{r_1}-\dfrac{r_1}{r_2}\right)\dif t^2+r_1r_2g,\\
\mathbf{\tilde A}^{\nu_j}=&-\Re\set{\hin{\dif\tilde F}{\dif\nu_j}}=r_1\mathbf{A}^{\sqrt{-1}\dif F(e_j)}+\dfrac{r_2}{r_1}\dif t\otimes\left(E_j\right)^{\sharp}+\dfrac{r_2}{r_1}\left(E_j\right)^{\sharp}\otimes\dif t,\quad j=2,\dotsc, n,\\
\mathbf{\tilde A}^{\sqrt{-1}\tilde F}=&0.
\end{align*}
We obtain that
\begin{itemize}
\item $\tilde F$ is CSL iff $F$ is CSL.
\item $\sqrt{-1}\mathbf{\tilde H}$ is parallel iff $\sqrt{-1}\mathbf{H}$ is parallel.
\item $\tilde F$ is minimal iff $F$ is minimal and $\abs{r_1}=\sqrt{\frac{n}{n+1}}$.
\end{itemize}

The second fundamental form can be written  by the matrix form as follows:
\begin{align*}
\mathbf{\tilde A}^{\nu_1}=\begin{pmatrix}\dfrac{r_2}{r_1}-\dfrac{r_1}{r_2}&0\\
0&\dfrac{r_2}{r_1}\mathrm{Id}_{(n-1)\times(n-1)}
\end{pmatrix},\quad\mathbf{\tilde A}^{\nu_j}=\begin{pmatrix}0&\dfrac{r_2}{r_1}\alpha_j^T\\
\dfrac{r_2}{r_1}\alpha_j&\dfrac{1}{r_1}\mathbf{A}^{\sqrt{-1}\dif F(e_j)}
\end{pmatrix},\quad \mathbf{\tilde A^{\sqrt{-1}\tilde F}}=0,\quad j=2,\dotsc, n,
\end{align*}
where
\begin{align*}
\begin{pmatrix}\alpha_2&\alpha_3&\dotsc&\alpha_{n}
\end{pmatrix}=\mathrm{Id}_{(n-1)\times(n-1)}.
\end{align*}
Hence,
\begin{align*}
\mathbf{\tilde{H}}_{\tilde g}=&\left(\dfrac{nr_2}{r_1}-\dfrac{r_1}{r_2}\right)\otimes\nu_1+\dfrac{1}{r_1^2}\mathbf{H},\\
\abs{\mathbf{\tilde{H}}}^2_{\tilde g}=&\left(\dfrac{nr_2}{r_1}-\dfrac{r_1}{r_2}\right)^2+\dfrac{1}{r_1^2}\abs{\mathbf{H}}_{g}^2,\\
\abs{\mathbf{\tilde B}}^2_{\tilde g}=&\left(\dfrac{r_2}{r_1}-\dfrac{r_1}{r_2}\right)^2+(n-1)\left(\dfrac{r_2}{r_1}\right)^2+2(n-1)\left(\dfrac{r_2}{r_1}\right)^2+\dfrac{1}{r_1^2}\abs{\mathbf{B}}_{g}^2.
\end{align*}

When $M$ is totally geodesic, a direct calculation yields
\begin{align*}
\abs{\mathbf{\tilde B}}_{\tilde g}^2\geq\dfrac{(n-1)(n+2)}{n}+\dfrac{n^2+3n-2}{2n^2}\abs{\mathbf{\tilde H}}_{\tilde g}^2-\dfrac{(n-1)(n-2)\abs{\mathbf{\tilde H}_{\tilde g}}\sqrt{4n+\abs{\mathbf{\tilde H}_{\tilde g}}^2}}{2n^2}.
\end{align*}
The equality holds if and only if $\abs{r_1}\leq\sqrt{\frac{n}{n+1}}$ or equivalently $\abs{r_2}\geq\sqrt{\frac{1}{n+1}}$. We also have
\begin{align*}
\abs{\mathbf{\tilde B}}^2_{\tilde g}-\dfrac{3n-2}{n^2}\abs{\mathbf{\tilde H}}_{\tilde g}^2-\dfrac{4(n-1)}{n}=\dfrac{(n-1)(n-2)}{n^2}\dfrac{r_1^2}{r_2^2}.
\end{align*}


\end{document}